\documentclass[10pt] {amsart}
\usepackage{amsmath}
\usepackage{amsfonts,amssymb,latexsym,xspace}
\usepackage{enumerate}
\usepackage[english]{babel}
\usepackage{graphicx}

\def\hat {\widehat}
\def\E {{\mathbb E}}

\def \Pr {\mathbb P}
\def\cal {\mathcal}
\def \x {\boldsymbol x}

\def \tilde{\widetilde}

\def\cal{\mathcal}
\def \cal  {\mathcal}
\def \w {{\boldsymbol w}}

 \newtheorem{theorem}{Theorem}
 \newtheorem{proposition}{Proposition}
 \newtheorem{lemma}[proposition]{Lemma}
  \newtheorem{definition}[proposition]{Definition}

\begin{document}

\title [The  Depoissonisation quintet :  Rice-Poisson-Mellin-Newton-Laplace]
 {The  Depoissonisation quintet:  Rice-Poisson-Mellin-Newton-Laplace } 

\author{Brigitte Vall\'ee}

\date{February 2018}
\maketitle

\begin{abstract} 
This paper  is  devoted to the Depoissonnisation process, which is central in  various analyses of the AofA domain.  We first recall the two possible  paths that  may be used in this  process. The first path, called here the Depoissonisation path, is better studied and  is proven to apply in any practical situation;  however,  it  often uses technical tools,  that are not so easy to deal with. Moreover,   the  various results are scattered in the litterature, and  the  most recent results are not  well known within the AofA domain. The present paper gathers in Section 2 all these results in a survey style. The second path, called here the Rice-Mellin path, is less often used within the AofA domain. It is often very easy to apply, but it needs a tameness condition, which appears {\em a priori} to be quite restrictive, and is not deeply studiedin the litterature. In Section 3, the paper  precisely describes the Rice-Mellin path, together with its tameness condition,  in a survey style, too. Finally, in Section 4, the paper presents original results for the Rice-Mellin path: it exhibits a framework,  of practical use,  where the tameness condition is proven to hold.  It then proves that the Rice-Mellin path  is both of  easy and practical use : even though (much?)  less general than the Depoissonisation path, it is  easier to apply. 
\end{abstract}

\bigskip
This is a long paper,  with  around twenty  five pages,  two times longer than the twelve pages abstract which is  recommended for the submission to the  AofA Conference. 
After a general introduction (in Section 1),  it contains two  main parts;   the first (long) part, with  Sections 2 and 3,   is of survey style;   the second part (Section 4) contains original results.

\smallskip
  When restricted to its  Sections 1 and 4,  the paper thus contains  around  fifteen pages; moreover,  most of  Section 4 can be read and understood essentially  with Section 1,  and only few  precise references to Sections 2 and 3. We could have chosen to submit this  shorter version, but  we finally choose to submit the present  long version. We indeed   think that  the survey part (Sections 2 and 3)  may be  useful for the AofA people, on two  main aspects.

  \smallskip First,  the  two paths are   rarely   described {for themselves} in the litterature,  and  general  {methodological} results  are often difficult to isolate amongst {particular} results that are more directed towards  various applications.   It  was  not an easy  task for us to collect 
  the main  results for the Depoissonisation path, as they  are scattered in at least five papers,  with a   chronological order  which does not correspond to the logical order of the method.  The Rice-Mellin path  is  also almost always presented  in the litterature with a strong focus towards possible applications.  Then, its use seems to be very restrictive, and this explains  why  it is  very often undervalued. 
  
  \smallskip Second, the two paths are not precisely compared, and the situation creates  various ``feelings'':  some people see the tools  that are used in  two paths as  quite different, and strongly prefer one of the two paths;  some other  think  the two paths are almost the same, with just a change of vocabulary.  It is thus useful to  compare the two paths when they may  be both used, and exactly compare their tools.  We perform this comparaison on a precise problem, related to the analysis of tries,  introduced in Section  \ref{secanatrie}. 
  
  \medskip
  Section 4  exhibits a general framework, the basic sequences (and even the extended basic sequences),  where the Rice-Mellin  conditions are proven to hold. This Section uses the shifting of sequences described in Section \ref{VD} \ and then the inverse Laplace transform, which does not seem of classical use in this context. This adds a new method to the Depoissonisation context and explains the title  of  our paper.  This approach  only deals  with integrals on the real line, and use quite simple tools. It  is perhaps of independent interest.

\section{General framework.} 
We first recall the two  probabilistic models, the Bernoulli model and the Poisson model. Then we introduce in Sections 1.3  to 1.5   the two main objects attached to a sequence $f$: the  classical Poisson transform $P_f$, and  another sequence,  denoted as $\Pi[f]$ and  called here the Poisson sequence. This terminology is not classical,  both for the name (Poisson sequence) and the notation (the mapping $\Pi$). We insist on the  involutive character of the mapping $\Pi$. Then, we introduce  the shifting operation  $T$ on the sequences, and observe that the two maps $\Pi$ and $T$ almost commute.  This leads to the notion of canonical sequence, which proves very useful in our study. Then,  Section 1.7  describes the two paths of interest. The  last three subsections  of  Section  1 are devoted to a particular analysis, the trie analysis,  which strongly motivates the present work, and is performed within each of the two paths.  

\subsection{Algorithms whose inputs are finite sequences of data} 
Many algorithms deal with inputs that are finite  sequences of  data. We  give some examples :  
$(a)$ 
for text algorithms,   data  are words, and  inputs  are   finite sequences of words;   $(b)$
for  geometric algorithms,  data  are  points,  and  inputs are  finite sequences of points; 
$(c)$ for a source,  data are symbols, and inputs are   finite sequences of symbols, namely  finite words.

 The cardinality  of the sequence plays  a prominent role, and  it is often chosen as the  input size. As usual, one is interested in the  asymptotic behaviour of the algorithm for large  size.

 \subsection{Probabilistic models.}
The  probabilistic framework is as follows: Each data (word or point) is  produced along a distribution, and the set  of  data is thus a probabilistic space $({\cal X}, {\mathbb P})$.  There  are various cases for the set ${\cal X}^n$ of  sequences of length $n$;  
 very often,  the data are independently chosen with the same distribution and the set $({\cal X}^n,  {\mathbb P}_{[n]})$ is  the product of order $n$ of the space $({\cal X}, {\mathbb P})$.

\medskip
The space of all the inputs is thus the set ${\cal X}^\star:=  \sum_{n \ge 0} {\cal X}^n$ of   finite sequences $\x$ of elements of ${\cal X}$, 
and there are two main probabilistic models:

\begin{itemize}
\item [$(i)$]  The Bernoulli model
${\cal B}_n$,  where the cardinality   of $\x$ is fixed and equal to  $n$ (then tends to $\infty$); 

\item[$(ii)$]    The Poisson model ${\cal P}_z$ of parameter $z$,  where the cardinality    of $\x$  is a random variable  that follows a Poisson law of parameter $z$,   
$$ \Pr [ |\x|  = n] = e^{- z} \frac {z^n}{n!}, $$ 
where the parameter  $z$ is fixed (then tends also to $\infty$).

\end{itemize} 

\subsection{Instances of natural  costs defined on ${\cal X}^\star$.} 
The Bernoulli model is more natural in algorithmics, but the Poisson model has  very nice probabilistic properties, notably properties of independence. Thus, one very often proceeds as follows: 
  one considers a  variable  (or a cost) $R: {\cal X}^\star \rightarrow {\mathbb N}$   which describes the behaviour of the algorithm  on the input;  for instance, for     $\x \in {\cal X}^\star$, 
  \begin{itemize}
  
  \item[$(a)$]    $R(\x)$ is the path length of a tree (trie or dst) built on the sequence $\x:= (x_1, \ldots, x_n)$ of words $x_i$ 
  
   \item[$(b)$]   $R(\x)$ is the number of  points  in the convex hull built   on the sequence  $\x= (x_1, \ldots, x_n)$
   of points $x_i$
   
  \item[$(c)$]  $R(\w)$ is the probability $p_\w$ of the word $\w$ viewed as a sequence  $\w := (w_1 \dots, w_n)$ of symbols $w_i$
  
  \end{itemize}   Our final aim is the analysis of $R$ in the model ${\cal B}_n$, 
  but we often begin to perform the analysis in the easier Poisson model  ${\cal P}_z$, 
  and we then wish to return in the Bernoulli model.

    \subsection{The Poisson transform and the Poisson sequence}   \label{Secdef}
 
  Our final aim is to study the asymptotics of 
  the sequence  $n \mapsto f(n)$, 
  where $f(n) := \E_{[n]}[R]$ is the expectation of the cost $R$ in the Bernoulli model ${\cal B}_n$.  
    The expectation $ \E_z[R]$ in the Poisson model ${\cal P}_z$ then satisfies
   $$  \E_z[R] = \sum_{n \ge 0} \E_z [R \mid\!\! N = n]  \,  \Pr_z [N= n]  
    	 = \sum_{n \ge 0} \E_{[n]} [R]  \,   \Pr_z [N= n] =  e^{-z}  \sum_{n \ge 0} {  f(n)} \frac {z^n}{n!} \, .$$
		
This leads us to introduce the Poisson transform  of the sequence $f : n\mapsto f(n)$,  
\begin{equation} \label{Pf} P_f(z) :=   e^{-z}  \sum_{n \ge 0} {  f(n)} \frac {z^n}{n!} \, , 
\end{equation}
and   the following holds: 

\begin{lemma} 	 Consider a cost $R$ defined on the set ${\cal X}^\star$,  its expectation  $f(n)$ in the Bernoulli model ${\cal B}_n$.  Then  its expectation $\E_z[R]$ in the Poisson model ${\cal P}_z$ coincides with the Poisson transform $P_f(z)$ of the sequence $f: n \mapsto f(n)$ defined in \eqref{Pf}.
\end{lemma}

The Poisson transform   itself can be written  as an exponential generating function (with ``signs'')\footnote{The  signs are added  in order to get an involutive formula in \eqref{binom}.}  and this defines a sequence $p: n \mapsto p(n)$,  
\begin{equation} \label{Pfp}
P_f(z) := e^{-z} \sum_{k \ge 0}  { f(k)} \, \frac {z^k} {k!} = \sum_{k \ge 0} (-1)^k \frac {z^k}{k!} { p(k)} \, 
\end{equation}
which will be called the Poisson sequence of the sequence $f$. We summarize:

\begin{definition}  Consider a sequence $f: n \mapsto f(n)$. 

$(a)$ 
The series  defined in \eqref{Pf}
is called  the {Poisson transform} of  the sequence $f$.

$(b)$ 
The sequence $p: k\mapsto p(k)$ defined  in \eqref{Pfp} by the signed coefficients of $P_f(z)$,  
\begin{equation} \label{p} p(k) := (-1) ^k  k ! [z^k]   P_f(z)\, 
\end{equation}
is called the Poisson sequence of the sequence $f$. It is denoted as $\Pi[f]$. 

$(c)$
Relation \eqref{Pfp} holds between the initial sequence $f$,  its Poisson transform  $P_f$ and its Poisson sequence $\Pi[f]$.  
\end{definition}

  \begin{lemma} 
  There  are  binomial relations between the sequences   $f$  and $p := \Pi[f]$, namely
  \begin{equation} \label{binom} 
   {p(n)} = \sum_{k = 0}^n (-1)^k {n \choose k} {f(k)}, 
  \quad \hbox{and}\quad{f(n)} = \sum_{k = 0}^n (-1)^k {n \choose k} { p(k)} .
  \end{equation}
  The map $\Pi : f \mapsto  \Pi[f]$  is involutive. 
  \end{lemma}

    \subsection {Shift $T$  and canonical  sequences} \label{VD}
     This technical section  which will be important in the sequel. The notions that are presented here are not  introduced in this way in the litterature, and,  in particular, the notion of canonical sequence appears to be new (and useful), notably in Section~4. 
  
\begin{definition}  \label{red} Consider  a non zero  real sequence  $n \mapsto f(n)$. 

$(a)$ Its degree   ${\rm deg} (f)$ and its valuation ${\rm val} (f)$ are defined as 
   $$   {\tt deg} (f):= \inf\{c \mid f(k) = O(k^c)\}   \qquad {\tt val} (f) := \min \{ k  \mid f(k)\not = 0\}\, .$$
  A sequence $f$  with finite degree is said to be of polynomial growth. 

$(b)$ A sequence $n \mapsto f(n)$  satisfies 
  the  Valuation-Degree Condition ({\tt VD}), 
 
 \vskip 0.1cm 
  \centerline{iff   
  ${\tt val} (f) >  {\tt deg } (f) +1$.}
  
  $(c)$ It is reduced   if it satisfies $  {\tt val} (f) = 0$ and  ${\tt deg } (f) <-1$.
  \end{definition}

 \bigskip
  This  {\tt VD} Condition  will be  important in the following proofs of Sections 3 and 4. However, as we are (only) interested in the asymptotics of the sequence $f$, the {\tt VD} condition   is easy to ensure, as we now show: 
 We begin with a sequence $F$  of polynomial growth with ${\tt deg}(F) =   d$, we  associate to $d$  the integer $\sigma(d) \ge 1$, defined   as follows: 
 \begin{equation} \label{ell}
 \sigma(d) := 1 \quad  \hbox {(if $d >0$)}, \quad  \sigma(d):= 2 + \lfloor d \rfloor  \quad \hbox{ (if $ d \ge 0$)}\,  . 
 \end{equation}
Then,   the inequality  $ \sigma(d)>d+1$ holds;  we replace the  first terms of the sequence $F$  with index $k < \sigma(d)$ by zeroes. We thus define the sequence $F^+$ as 
   $$F^+(n) = 0  \quad \hbox{for 
   $n \le  \sigma(d) -1$}, \qquad  F^+ (\sigma(d))  = 1, \qquad   F^+ (n)  = f(n)  \quad \hbox {for 
  $ n >  \sigma(d)$} \, .$$
  As ${\tt deg}(F^+)  =   {\tt deg}(F) = d $ and ${\tt val} (F^+) = \sigma(d)$ with $\sigma$ defined in \eqref{ell}, this entails  the inequality ${\tt val} (F^+) >  {\tt deg } (F^+) +1$:   the sequence $F^+$  keeps  the same asymptotics as the  initial sequence $f$ and now satisfies the {\tt VD} condition.

\medskip 
Start now with  a sequence $f$  of degree $d$ and of valuation 
${\tt val} (f) = \ell = \sigma(d)$.
Then the Poisson transform $P_f(z)$  has itself valuation $\ell$ and is  written as  
\begin{equation} \label {Q}
P_f(z) =  z^{\ell} Q(z) \quad \hbox{with} \quad 
  Q(z)  = e^{-z}   \sum_{ k \ge 0}   g(k)\,    \frac {z^{k} }{ k !}  = \sum_{k \ge 0} (-1)^k \frac {z^k}{k!} {q(k)} \, .
\end{equation}
We now express the new sequences $g$ and  $q = \Pi[g]$ (both with  zero valuation) in terms of the initial sequences $f$ and $p := \Pi[f]$.

\begin{lemma} \label{shift}  Consider the shifting map  $T$ which associates with a sequence $f$  the sequence $T(f)$ defined as  
$$  T[f] (n)    = \frac {f(n+1)}{n + 1}\, , \qquad   \hbox{and then} \quad  T^{m}[f] (n)   =   \frac {f(n+m)}{(n + 1)\ldots (n+ m)}\, ,  $$ 
for any $n \ge 0$ and any 
$m \ge 1$.   For $m \ge 1$, the inverse mapping  $T^{-m}$ associates with a sequence $g$ the sequence  $f$ defined as $$ f(n) = n(n-1) \ldots (n-m +1) \, g(n-m), \quad \hbox{ for $n \ge m$}\, .$$

$(a)$ The shifting $T$  ``almost'' commutes with the involution $\Pi$, 
 $$ T\circ \Pi =  - \Pi\circ T , \qquad   \Pi\circ T^{m}  =  (- 1)^{m}  T^{m} \circ \Pi \, .$$

$(b)$ Consider a sequence $f$ with  ${\tt val} (f) = \ell$. Then, the two sequences $g$ and $q:= \Pi[g]$ associated with $f$ via Eqn \eqref{Q} are expressed with  the  iterate of $T$ of order $\ell$, namely
$$ g  = T^{\ell} [f], \qquad q :=   
 \Pi[g] = (-1)^{\ell}  T^{\ell} [\Pi [f]]  $$ 
 
 $(c)$  If $f$   has degree $d$ and  valuation $\ell = \sigma(d)$, then
 the sequence $g$  defined in Eqn \eqref{Q} satisfies the {\tt VD} condition   $ {\tt val} (g) = 0$ and  ${\tt deg } (g)  = d - \sigma(d) <-1$; the sequence $g$ is reduced.

\end{lemma}

\begin{definition}  From an initial sequence $F$ of degree $d$,  and $\ell:= \sigma(d)$ as in \eqref{ell}, the reduced sequence $f:= T^\ell[F^+]$  is called the canonical sequence associated to $F$. It has zero-valuation, and  its degree equals  $c = d- \sigma(d)$. 
\end{definition}

Now, if the sequence $F$ admits an analytic extension on the halfplane $\Re s >0$ of polynomial growth there, its canonical sequence 
$f$  admits an analytic extension  $\varphi $ on the halfplane $\Re s >-\ell$,  which moreover satisfies $\varphi(s)  = O(|s+1|^c)$ there, with $c<-1$.  Then  $\varphi$ is integrable on each vertical line $\Re s = b$ with $b \in ]-1, 0[$. 

\smallskip 
In the sequel, it is then {\em sufficient} to deal with {\em canonical} sequences, and their Poisson pair. Then, 
 the results we obtain for the asymptotics on $f$   will be easily will be easily  transfered on the initial  Poisson pair of $F$ with Properties $(a)$ and $(b)$.  

\bigskip 
{\bf Example.}  In Section \ref{tries}, we  will be interested in the following  sequences $F_0, F_1,  F_2$,  all  of valuation 2, which satisfy moreover
$$ F_0 (k) = 1, \quad F_1(k) = k, \quad F_2(k) = k \log k , \qquad \hbox {for $k \ge 2$} \, .$$
The  sequence $F_0$ satisfies the {\tt VD} Condition, but not the two other ones, that we modify into the $F_1^+ $ and $F_2^+$ sequences. Finally, the canonical sequences are defined for $k \ge 0$, as 
$$ f_0 (k)= f_1(k) = \frac 1 {(k+1)(k +2)}, \qquad f_2 (k) =  \frac {\log (k+3)} {(k+1)(k+2)} \, .$$

  \subsection{Some definitions.}  \label{def}
 This section gathers various definitions and notations about domains of the plane, and behaviours of functions. 
   
   \smallskip{\bf \em Cones and vertical strips.} There are two  important types of domains of the complex plane  we deal with. 
   
  $(i)$  The  cones built on the real line ${\mathbb R}^+$,  with two possible definitions, 
  \begin{equation} \label{cones}
  {\cal C}(\theta) := \{ z \mid |\arg z | <\theta \} \  \hbox{ for $\theta<\pi$}, \quad \hat {\cal C}(\gamma) = \{ z \mid \Re z  > \gamma |z| \} \  \hbox {for $|\gamma| \le 1$} \, , 
  \end{equation}
  related by the relation $\hat {\cal C} (\cos \theta) = {\cal C}(\theta)$. 

\smallskip  
  $(ii)$  The vertical strips, or halfplanes
  $$ {\cal S} (a, b) := \{ z \mid a < \Re z < b \}, \quad { \cal S} (a) :=  \{ z \mid  \Re z >a \} \, .$$
    
  \medskip {\bf \em Polynomial growth.} 
This notion   plays a fundamental role: 
  
  \begin{definition}{\rm[Polynomial growth]} A function $s \mapsto \varpi(s) $  defined in an unbounded domain $\Omega \subset {\mathbb C}$ is said to  be of  polynomial growth if 
there exists $r$ for which   the  estimate $|\varpi(s)|=O(|s|^r)$ holds as $s\to\infty$ on $\Omega$.
\end{definition}

  When $\Omega $ is  included in a vertical strip ${\cal S}(a, b)$, this means:  
$|\varpi(s)|=O(|\Im s|^r)$; 
when $\Omega$  is included in a horizontal cone ${\cal C}(\theta)$ with $\theta < \pi/2$, this means:  
$|\varpi(s)|=O(|\Re s |^r)$; 

\medskip
{\bf \em Tameness.} In an informal way, the notion of tameness describes the behaviour of an analytic  function $\varpi$ on the left of  the halfplane $\Re s >c$  when it  stops being analytic. (see \cite{ClNTVa} for more precisions).

\begin{definition}{\rm[Tameness]} 
 A function  $\varpi$ analytic  and of polynomial growth on   $\Re s >c$  is  {tame} at $s = c$ 
if one of the three following properties holds:

\smallskip
$(a)$
{\rm [$S$-shape] (shorthand for Strip shape)} there
exists a vertical strip $\Re(s)>c- \delta$ for some $\delta>0$
where
$\varpi(s)$ is meromorphic, has a sole  pole (of order $b+1 \ge 1$)
at~$s=c$ and is of polynomial growth as $\lvert \Im s\rvert \to+\infty$.

\smallskip
$(b)$
{\rm [$H$-shape] (shorthand for Hyperbolic shape)} there exists an hyperbolic region
$\mathcal{R}$, defined as, for some $A, B, \rho >0$
\[
\mathcal{R}:= \{ s = \sigma +it ; \ \ \lvert t \rvert \ge B, \ \ \sigma> c- \frac {A}{\lvert{t}\rvert^\rho} \} \bigcup \{ s = \sigma +it ; \ \ \sigma>
c- \frac {A}{B^\rho}, \lvert t \rvert \le B \},
\]
where $\varpi (s)$ is meromorphic, with a sole  pole (of order $b+1$)
at~$s=c$ and is of polynomial growth in~$\mathcal{R}$ as $\lvert \Im s \rvert \to+\infty$.

\smallskip
$(c)$
{\rm [$P$-shape] (shorthand for Periodic shape)} there exists a vertical strip $\Re(s)>c- \delta$ for some $\delta >0$ where $\varpi(s)$ is meromorphic,
has only a pole
(of order $ b+1\ge 1$)
at~$s=c$ and a family $(s_k)$ (for $k \in \mathbb Z \setminus\{0\}$) of simple poles
at~ points $s_k= c+ 2ki \pi t$ with $t \not = 0$, and is of polynomial
growth as $\lvert \Im s \rvert \to+\infty$\footnote{More precisely, this means that $\varpi(s)$ is of polynomial growth on a family of horizontal lines $t = t_k$ with $t_k \to \infty$, and on vertical lines
$\Re(s)= \sigma_0- \delta'$ with some $\delta'< \delta$.}.

\end{definition}

  \subsection{Description  of the two possible paths.}  The sequence $f$  is often  given in an implicit way, and we only deal here   with sequences $f$ of polynomial growth: there exists $r \in {\mathbb R}$ for which $f(n) = O(n^r)$. Then its Poisson transform $z \mapsto P_f(z)$  is entire. We assume that we have some knowledge of type  $(a)$ {\em or} type $(b)$: 
  
  \begin{itemize}
 
 \item[$(a)$]  
    about the Poisson transform  $ P_f(z)$,

   \item[$(b)$]  
    about  the sequence  $\Pi[f]$.  
    
 \end{itemize}  
 The main question  is here:  
 \begin{quote}  {\em 
    Is it possible to  return to the initial sequence $f$ and    obtain  some knowledge about its asymptotics? }
    
     \end {quote}

We now describe the two paths.

 \medskip
$(a)$ {\bf \em Depoissonisation Path.} We deal with   the function $P_f(z)$. We assume  the following conditions  to hold, described as  ${\cal {JS}}$ [Jacquet-Szpankowski] conditions: 
 
 \smallskip
\noindent
\begin{quote} {\em The  asymptotic behaviour of $P_f(z)$ is well-known inside (and outside)  horizontal cones.}
\end{quote}

\smallskip
\noindent
Then,  it is possible to  precisely  compare the asymptotics of $P_f(z)$ (when  $z$  tends to  $\infty$  inside   horizontal cones) and the asymptotics  of the sequence $n \mapsto f(n)$. This is the Depoissonization path.   

\smallskip
\noindent
Moreover, there exists  a condition on the input sequence $f$  under which the ${\cal {JS}}$ conditions hold: 
\begin{quote} 
{\em There exists  an analytic lifting $\varphi$  for the sequence $f$  
which  is of polynomial growth  inside   horizontal cones.  }
\end{quote}

\medskip 
$(b)$ {\bf \em Rice-Mellin Path.} We deal with the Poisson  sequence  $\Pi[f]$. 
We assume  the following conditions to hold,  described as  ${\cal {RM}}$ [Rice-Mellin]  conditions:    

\smallskip
\noindent
\begin{quote}
{\em  There exists an analytic lifting $\psi(s)$ for the sequence $\Pi[f]$  inside vertical strips,  which  is of polynomial growth  and tame. }
\end{quote}

\smallskip
\noindent
 Then the binomial recurrence  \eqref{binom} is transfered into a relation which expresses the  term $f(n)$ as an integral along a vertical line which involves the analytic lifting $\psi(s)$. With tameness of $\psi$, we obtain the asymptotics of the sequence $f$.  This is the Rice-Mellin method.
 
 \smallskip
\noindent
 However, the conditions  on the input sequence $f$ under which the ${\cal {RM}}$ conditions hold on the analytical lifting $\psi$ are not  clearly described  in the litterature. This is the main  purpose of the present paper.

  \subsection{A transversal tool: the Mellin transform.} 
 We recall that the Mellin transform of a function $Q$ defined in $[0, + \infty]$ is defined as
$$Q^\star (s) := \int_0^{+ \infty}  Q(u) \, u^{s-1} du\, .$$
The Mellin transform plays a central role in each of the two paths. Its properties are very well described in  the  survey paper \cite{FlGoDu}  on Mellin transforms.    In particular, we need the good behaviour of the transform on harmonic sums (see  next Section  \ref{tries}) and also 
the  following lemma\footnote{It is called the  Exponential Smallness Lemma in the paper \cite{HwFuZa}, and we keep the same  terminology.} which  proves  that the function $\Gamma(s)$ and its derivatives  $\Gamma^{(m)}(s)$ are  exponentially small along vertical lines (when $ |\Im (s)| \to \infty$).  

\begin {lemma} \label{ES}{\rm [Exponential Smallness Lemma]  \cite{FlGoDu}} If, inside the  cone $\overline {\cal C}(\theta)$  with $\theta>0$ one has $Q(z) = O(|z|^{-\alpha})$ as $z \to 0$ and $Q(z) = O(|z|^{- \beta}$ as $|z| \to \infty$, then  the estimate $Q^\ast(s)  = O( \exp[-\theta |\Im (s)|])$ uniformly holds in the vertical strip ${\cal S} ( \alpha, \beta)$.
\end{lemma}

  \subsection {An instance of  Depoissonisation context. Probabilistic analysis of  tries} \label{tries}
  A source ${\cal S}$  is a  probabilistic process 
which produces infinite  words  on the (finite)  alphabet $\Sigma:= [0 .. r-1]$. A trie is a  tree structure,  used as a dictionary, which compares words  via their prefixes. Given
a finite  sequence  $\x= (x_1, x_2, \ldots , x_n)$  formed with $n$  (infinite) words emitted by the source ${\cal S}$,  the trie
  ${\cal T}(\x)$  built on the sequence\footnote {The trie depends only on the underlying set $\{x_1, x_2, \ldots , x_n\}$.}
 $\x$  is defined recursively
by the following three  rules:

\begin{itemize}

\item [$(i)$] \ \   If $|{\x}| =0$, ${\cal T}(\x)=\emptyset$

 \item [$(ii)$] \ \  If $|\x|=1,\ \x=\{\ \! x\}$, ${\cal T}(\x)$ is a  leaf labeled
by $x$.

 \item [$(iii)$]  \ \ If $|\x| \ge 2$, then
 ${\cal T}(\x)$ is formed with an internal node and $r$ subtries respectively equal to
 $${\cal T}( \x_{\langle 0\rangle}),\ldots ,{\cal T}(\x_ {\langle  r-1\rangle})$$
 where $\x_{\langle \sigma \rangle}$ denotes  the set consisting of
 words of $\x$ which begin  with symbol $\sigma$,  stripped of their initial
symbol $\sigma$. If the set  $\x_{\langle \sigma \rangle}$ is non empty, the edge which links the subtrie
${\cal T}(\x_{\langle \sigma \rangle})$ to the  internal node is labelled with the symbol $\sigma$.

\end{itemize}
Then,  the internal nodes are used for directing the search, and the leaves contain  suffixes  of $\x$. There are as many leaves as  words in $\x$.  The internal nodes are labelled by prefixes $\w$ for which the  cardinality $N_\w$ of the  subset $\x_{\langle \w\rangle}$  is at least  2.

\medskip  
 Trie analysis aims at describing the average shape of a trie (for instance: number  of internal nodes $S$, external path length $P$, height $H$, etc....). 
 We focus here on {\em additive} parameters,  whose (recursive) definition exactly copies the (recursive) definition of the trie.  With  a   sequence $f: {\mathbb N} \rightarrow {\mathbb R}$ -- called a {\em toll}--  which satisfies $f(0) = f(1) = 0$ and $f(k) \ge 0$ for $k \ge 2$,   we associate  a random variable  $R$ defined on the  set ${\cal X}^\star$  
 as follows: 
 
 \begin{itemize}

\item [$(i)$] \ \   If $|\x| \le 1$,  then \quad $R(\x)=0$

 \item [$(ii)$] \ \  If $|\x| \ge 2$, then \quad 
 $\displaystyle   R(\x)= f(|\x|) + \sum_{\sigma  \in \Sigma}  R(\x_{\langle  \sigma\rangle})$.

\end{itemize}   Iterating the recursion leads to the expression
\begin{equation} \label{SK1}
R (\x):= \sum_{\w \in \Sigma^\star} f(N_\w( \x)), 
\end{equation}
where $N_\w(\x) $ is the number of words  in $\x$ which begin with the (finite) prefix $\w$. 

\smallskip
The size is associated to the toll $f(k) = 1$ (for $k \ge 2$) and the path length  to the toll $f(k) = k$ (for $k \ge 2$).    A version of  the {\tt QuickSort} algorithm   on words  \cite {ClNTVa} leads  to the toll
 $f(k) = k \log k$ (for $k \ge 2$) that we call in the sequel the {\em sorting toll}. 
 
\smallskip
 We are interested in  the asymptotics of 
the mean value $r(n)$ of the random variable $ R$ in the Bernoulli model ${\cal B}_n$, 
 (as $n \to \infty$).  Of course, the probabilistic behaviour  of $R$  will depend  both on  the toll $f$ and 
the source ${\cal S}$. The  probabilistic properties of  the source ${\cal S}$ are themselves
defined from the fundamental probabilities 
\begin{equation} \label{piw}
 \pi_{\w} = \Pr [\hbox{a word  emitted by ${\cal S}$  begins with the prefix $\w$}] \, , 
 \end{equation}
and summarized by the analytic properties of the  generating Dirichlet series  
\begin{equation} \label{Lambda}
 \Lambda(s) := \sum_{\w \in \Sigma^\star} \pi_{\w}^s\, , 
 \end{equation} 
 introduced in \cite {Va1}  and called  there the Dirichlet series of the source.

\subsection{Main principles of  trie analysis}
 The main advantage of the Poisson model  in the framework of sources is the following: In the Poisson model of rate $z$, the variable $N_\w$  which appears in Eqn \eqref{SK1} follows a Poisson law of rate $z\,  \pi_\w$ and involves the  fundamental probability $\pi_{\w}$ defined in \eqref{piw}.  We  then adapt the general framework defined in Subsection \ref{def},  both for the initial sequence $f$ and for the sequence $r$, and consider the two paths:    
 
 \smallskip
 -- in Path $(a)$,  we deal with the Poisson transforms $P_r(z)$ and $P_f(z)$. Then,   averaging Relation (\ref{SK1}) in the Poisson model of rate $z$ entails a relation between  the  two Poisson transforms \begin{equation} \label{SK2} P_r(z) = \sum_{\w \in \Sigma^\star} \E_z [f(N_\w)]= \sum_{\w \in \Sigma^\star} P_f(z\,  \pi_\w). \end{equation} 

\smallskip 
-- in Path $(b)$, we deal with the Poisson sequences $q= \Pi[r]$ and $p = \Pi[f] $.  Then, Relation \eqref{SK2}  
entails the equality
\begin{equation}
q(n) = 
\Lambda(n) \, p(n),  \qquad \hbox{with} \quad  \Lambda(s) := \sum_{w \in \Sigma^\star}\pi_w^s\, .
 \end{equation}
 
 \medskip{\bf Remark. } The role of the Dirichlet series $\Lambda(s)$ is clear  in Path $(b)$. However, the Dirichlet series $\Lambda(s)$   also  clearly appears in Path $(a)$  when using the Mellin transform. 
 Relation (\ref{SK2}) shows that  the function $P_r(z)$ is an harmonic sum\footnote{The function $G$ is an harmonic sum with base function $g$ and frequencies $\mu_k$ if 
 $ G(z)$ is written as $G(z) = \sum_{k} g(\mu_k z)$.} with base function $P_f$ and frequencies $\pi_{\w}$. With classical properties  of the Mellin transform \cite{FlGoDu}, its Mellin transform $P_r^\ast (s)$ factorises as
\begin{equation} \label{Gstar}
P_r^\ast (s) = \Lambda( -s) \cdot P_f^\ast (s)\, , \qquad \hbox{with} \quad  \Lambda(s)  := \sum_{\w \in \Sigma^\star} \pi_{\w}^{s} \, .
\end{equation} 

\medskip
The  probabilistic properties of the source  are described by  the behaviour of its Dirichlet series $\Lambda(s)$ defined in \eqref{Lambda}, notably near $s = 1$.  We consider here a {\em tame} source, for which  $s \mapsto \Lambda(s)$ is  tame at $s = 1$, with  a  simple pole  at $s = 1$ whose residue equals $1/h({\cal S})$ where  $h({\cal S})$ is the entropy of the source. (See \cite{ClNTVa} for a discussion about tameness of sources.)

\subsection {A precise result in the trie analysis.} \label{secanatrie}
The most classical tolls are related to the size of the trie  (with  $f(k) = 1$) and the path length (with $f(k) = k$). We focus here  on   the ``sorting toll'', defined as  \begin{equation} \label {sortoll}
f(k) = k \log k, \quad (k \ge 2), \qquad f(0) = f(1) = 0\, .
\end{equation}
and  are interested in the analysis of the  associated cost $R$, as a kind of test for comparing the two paths.   The analysis was  already performed  in  \cite{ClNTVa} with  Depoissonisation Path  $(a)$. We would  have wished there  to use   the Rice-Mellin  Path $(b)$ (as   we got used in our  previous analyses)  but we did not succeed in proving  the  ${\cal {RM}}$ Conditions to hold. 
This failure was a strong motivation for the present study.  We  now present in this paper two proofs for  the following  result,  each of them using one path. 

\begin{theorem}  \label{anatrie}Consider a trie built on  $n$ words emitted by a source ${\cal S}$.  Assume  furthermore that the Dirichlet series $\Lambda(s)$ of the source is tame at $s = 1$. Then the mean value   of  parameter $R$  associated with the  sorting toll $f$ defined in \eqref{sortoll}  satisfies in the Bernoulli model ${\cal B}_n$ 
$$ r(n)  \sim  \,  \frac 1 { 2\, h({\cal S})} \, n \log ^{2} n   \qquad (n \to \infty) \, .$$
\end{theorem}

\section {The  Depoissonization path.}  \label{DePoi}

The  Depoissonization path deals with the Poisson transform $P_f(z)$,  and performs the following steps: 

  \begin{itemize} 
  \item [$(1)$] It  compares $f(n)$ and $P_f(n)$ with the {Poisson--Charlier} expansion
  
  \item  [$(2)$] It  uses the {Mellin inverse} transform  for the asymptotics  of $P_f(n)$
  
  \item  [$(3)$]  It needs  {depoissonization}  sufficient conditions ${\cal JS}$,  
 for 
 {truncating} the Poisson-Charlier expansion
 
  \item [$(4)$] It 
  obtains the asymptotics of $f(n)$.

   \end{itemize}
   
   The main result is  informally described as follows:

\begin{theorem}  {Assume that  the  sequence $f$ admits an analytical lifting $\varphi(z)$ 
  on the half plane $\Re s >-1$, and is  of polynomial growth in a cone  ${\cal C}(-1, \theta_0)$ for some $\theta_0>0$.  
 Then the  Depoissonisation path can be applied :  the truncation of the Poisson-Charlier expansion gives rise to an  asymptotic estimate of the sequence $n \mapsto f(n)$ with ``good'' remainder terms. } 
\end{theorem}

It  is based on  five main contributions, that are scattered in the litterature.  The present  survey describes the main steps in a logical way, whereas the ideas and the proofs have not been always obtained in a chronological order.

 \smallskip  The Depoissonisation method, together  with its name, was introduced  in 1998 by Jacquet and Szpankowski in \cite{JaSz}.    They compare  the asymptotics of the  two sequences, the sequence $f(n)$ and the sequence $P_f(n)$.  There were surely  previous results of the same vein, but they were not known by the AofA community. Jacquet and Szpankowski did not use the Poisson-Charlier expansion (described in Section \ref{PoCh})  which was  later introduced  in 2010 into the AofA domain by  Hwang, Fuchs and Zacharovas  in  \cite{HwFuZa}.  Jacquet and Szpankowski also introduced  conditions on the Poisson transform that we call  (following the proposal of   \cite{HwFuZa}) the ${\cal {JS}}$ conditions, described here in Section \ref{JS}. In fact, similar conditions may be found in earlier papers, notably  a paper due to Hayman \cite{Ha} in 1956.  In \cite{JaSz}, the authors prove that, under  ${\cal {JS}}$ conditions, it is possible to compare the two sequences $P_f(n)$ and $f(n)$. Later on, in 2010, using the Poisson Charlier expansion, the authors of  \cite{HwFuZa}  obtain  a  direct and natural proof of this comparison, with a more  explicit remainder term (See Theorem~\ref{thm1}).

  \smallskip
  Finally, in two  other papers,  Jacquet and Szpankowski discussed necessary and sufficient conditions on the initial sequence $f$   for  ${\cal {JS}}$ conditions to hold on $P_f$. The paper  \cite{JaSz1}  deals with  the necessary condition [see Theorem~\ref{thm2} $(ii) \Longrightarrow (i)$]  whereas the  very recent paper \cite {Ja} deals with the sufficient condition [see Theorem~\ref{thm2} $(i) \Longrightarrow (ii)$].

   \subsection{The Charlier-Poisson expansion. } \label{PoCh}
   It was introduced into the AofA domain by   Hwang, Fuchs and Zacharovas  in  \cite{HwFuZa}. One begins with the Taylor expansion of $ P(z) := P_f(z)$ at $z = n$, namely 
 $$P(z) = \sum_{j \ge 0}  \frac {P^{(j)}(n)} {j!} \, {(z-n)^j} . $$
As  the sequence $n \mapsto f(n)$ is of polynomial growth,  the function $z \mapsto e^z P(z)$ is entire, and  there are two expressions of 
$$  e^z P(z) =  \sum_{n \ge 0} \frac {z^n}{n!} \, f(n)  = \sum_{j \ge 0}  \frac {P^{(j)}(n)} {j!} \, e^z{(z-n)^j} . $$ 
The   {Charlier-Poisson} sequence $\tau_j$  is then  related to  
the coefficient of order $n$ in the expansion of $z \mapsto e^z  (z-n)^j$, 
$$ \tau_j(n) :=  n! [z^n] \left((z-n)^j e^z\right)  =  \sum_{\ell = 0}^j  {j \choose \ell} (-1)^{j-\ell}  n^{j-\ell}  \frac {n!}{(n-\ell)!}\, ,  $$
and  is closely related to the  (classical) Charlier polynomial. It   
 is  itself a polynomial in $n$ of degree $\lfloor j/2 \rfloor$.
Then there is an  (infinite) expansion of  $f(n)$, 
$$ 
 f(n) := n! [z^n] \left( e^z P(z)\right) =  \sum_{j \ge 0}  \frac {P^{(j)}(n)} {j!}  \tau_j(n)\, . $$
which is always valid. But we wish {\em truncate}  the infinite sum in order to obtain an estimate of $f(n)$.   
 What happens when we  only keep the first  terms?  Which error is expected?  
  We need  here 
  conditions  on the  Poisson transform $P(z)$ in cones. 
 
\subsection{$ {\cal {JS}}$ Conditions for depoissonisation} \label{JS}
 There  are  {sufficient} conditions on the behaviour of the Poisson transform in {cones}  first described by Haymann (1956) in \cite{Ha}
 and introduced into the AofA domain by Jacquet and Szpankowski (1998) in \cite {JaSz}. 

\begin{definition}  
{\rm  [${\cal {JS}}$ admissibility]}  {
 An   entire function $P(z)$  is $ {\cal {JS}}$-admissible with parameters  $(\alpha, \beta)$ if there exist
  $\theta \in ]0, \pi/2[ $,     $\delta  <1$ for which (for $z \to \infty$)
\begin{itemize}
\item[$(I)$]
 Inside cone ${\cal C}(\theta)$,   one has $  |P(z)|  = O \left( |z|^{\alpha} \log^{\beta} (1 + |z|) \right) $.
\item[$(O)$]
 Outside cone  ${\cal C}(\theta)$, one has 
${ |P(z) e^z|} = O \left(  e^{{\delta} |z|)} \right) $.
\end{itemize}
}
\end{definition}

\begin {theorem} \cite{JaSz, HwFuZa} \label{thm1}
 If    the Poisson transform  $P_f(z)$ of  the sequence $f$  is ${\cal {JS}}(\alpha, \beta)$ admissible,  then
 the first terms of the Poisson-Charlier expansion  provide  the beginning of the  asymptotic expansion  of $f(n)$. More precisely, for any $k>0$, one has:  
$$ f(n) =   \sum_{0\le j <2k}   P^{(j)} (n) \,  \frac {\tau_j (n)} {j!} + O( n^{\alpha-k} \log^\beta n) \, . $$
\end{theorem}

 \begin{proof} {\rm [Sketch of the proof] \cite{HwFuZa}}.  Starting from Cauchy's integral formula,  
 \begin{equation} \label{depoi-i}
  f(n) = \frac {n!}{2 i \pi} \int_{|z] = r}  P(z)\,  e^z  \frac 1 {z^{n+1}}  dz  \, ,
\end{equation}
the result follows  from a standard application of the saddle-point method. Condition $(O)$ guarantees that the integral over the circle with radius $n$ outside Cone ${\cal C}(\theta)$ is negligible, while Condition $(I)$ entails smooth estimates for all derivatives (and thus error terms).
  \end{proof}

\subsection{Characterization of the ${\cal {JS}}$ admissibility.} The following result is important as it  provides a	 characterization of the ${\cal {JS}}$ conditions  on   the initial sequence $f$ itself.

\begin{theorem}  \cite{JaSz1, Ja} 
 \label{thm2} {Let  $f$ a sequence of polynomial growth and its  Poisson transform $P_f$.
The two conditions are equivalent: 
\begin{itemize} 
\item [$(i)$]  $P_f(z)$  is  ${\cal {JS}}$-admissible
 \item [$(ii)$]  The  sequence $f$ admits an analytical lifting $\varphi(z)$ 
 on the half plane $\Re s >-1$  
 of polynomial growth in a cone  ${\cal C}(-1, \theta_0)$ for some $\theta_0>0$. 
 \end{itemize}
}
\end{theorem}

\subsection {Elements of proofs for Theorem \ref{thm2} } We describe  the main principles that are used in the proofs.

 \paragraph {$(i) \Longrightarrow (ii)$}
 Jacquet begins  in \cite{Ja} with the Cauchy Formula that provides an integral expression for $f(n)$; he then obtains an extension $\varphi$ of the sequence $f$, as
 $$\varphi(s-1) =  \frac{ \Gamma(s+1)}{2 \pi}  s^{-s} e^s  \int_{-\pi}^{+ \pi}  \, P(s e^{i\theta} )\, e^{i \theta} \exp[ s(e^{i\theta} - 1- i\theta)] \, d \theta$$
 which is analytic   when $s$ belongs to any cone ${\cal C}(\theta)$ 
 with  an angle $\theta <\pi$. Then,  the ${\cal {JS}}$ conditions  entail the polynomial growth of $\varphi$  when $s$ belongs to a cone  ${\cal C}(\theta_0)$ with  some $\theta_0 >0$. 

\medskip
\paragraph {$(ii) \Longrightarrow (i)$} 
 The proof of \cite{JaSz1}  uses the Laplace transform. As far as we know, this is the first occurrence of this transform in the Depoissonisation context, and this is  also important for us, in the present context, since we will use the (inverse) Laplace transform in Section \ref{RiceLap}.  \\    With a function  $f: [0, + \infty[ \rightarrow {\mathbb C}$ of polynomial growth,   the  Laplace transform associates the function $\widetilde f$,  defined on the halfplane  $\Re s >0$  via the relation,  
 $$\widetilde f(s) := \int_0^{+ \infty}  f(x) \, e^{-sx}  \, dx \, , $$
  that is analytic there.  Now, the authors of \cite{JaSz1}  use two main  results:

  \begin{lemma} Consider a function  $f: [0, + \infty[ \rightarrow {\mathbb C}$ that admits an analytic continuation $ \varphi$ to a cone ${\cal C} (\theta_0)$ with $\theta_0 < \pi/2$ on which it is of polynomial growth.  Then, the Laplace transform $\widetilde \varphi$ of $\varphi$ admits an analytic  continuation on the cone ${\cal C} (\theta_0 + (\pi/2))$
  \end{lemma}
  
  \begin{lemma}  The  Poisson transform $P_f(z)$ is expressed with the Laplace transform of the analytic continuation $\varphi $ of the sequence $f$ under the form 
$$ P_f(z)  = \frac 1 {2 i \pi}  \int_{\cal L} \widetilde \varphi(s) \exp[  z (e^s -1)] ds $$
where ${\cal L}$ is included in the cone ${\cal C}(\theta_0 + (\pi/2))$.
\end{lemma}

Then, the authors in \cite{JaSz1} use as the contour ${\cal L}$ a curve that parallels the boundary of the cone 
${\cal C}(\theta_0 + (\pi/2))$. They study the behaviour of the function $A: (s, \theta) \mapsto \Re (e^s e^{i\theta})$ and prove that $A(s, \theta)< \alpha <1$ when $s \in {\cal L}$ and $|\theta|  \le \theta_0$ for some $\theta_0 >0$.  This ends the proof of  Theorem \ref{thm2}  {$(ii) \Longrightarrow (i)$}.

\subsection{Application to the sorting toll in tries. First proof of Theorem \ref{anatrie}.}  \label{triedepo} 
This section 
 ends with an example of application of the Depoissonisation path  to the study of trie parameters.
  We then obtain a first proof of Theorem \ref{anatrie}, using the Depoissonisation Path.
 
 \smallskip
We begin with Relation \eqref{Gstar}, 
$$P_r^\ast (s) = \Lambda( -s) \cdot P_f^\ast (s)\, , \qquad \hbox{with} \quad  \Lambda(s)  := \sum_{\w \in \Sigma^\star} \pi_{\w}^{s} \, .$$
On the other side,  one has
$$
P_f^\ast (s) = \sum_{k\ge 2 } \frac { f(k)} {k!} \int_0 ^\infty e^{-z} z^{k} z^{s-1} dz
 = \sum_{k\ge 2} \frac {f(k)}{k!} \Gamma(k +s) 
= \sum_{k\ge 2} \frac {f(k)}{k} \, \frac { \Gamma(k +s)}{\Gamma (k)}.
$$
The ratio of Gamma Functions can be estimated with the Stirling Formula,
\begin{equation} \label{Stirling}
\frac { \Gamma(k +s)}{\Gamma (k)} = \frac { (k +s)^{k +s} }{ k^k} \, \frac {e^{-k -s}}{e^{-k}} \, \sqrt { \frac {k +s}{k}} \left[ 1 + O\left( \frac 1 k\right) \right] = k^s \left[ 1 + O\left( \frac {\lvert s \rvert} k\right) \right],
\end{equation}
where the $O$-term is uniform with respect to $k$. Then, the Mellin transform of $P_f$ satisfies, for 
$f(k) = k \log k$, 
\begin{equation} \label {Fast} 
P_f^\ast(s) = \sum_{k \ge 2} k^{s} \log k \left[ 1 + O\left( \frac {\lvert s \rvert} k\right) \right] = -\zeta'(-s) + H_1(s), \end{equation}
 where $H_1(s)$ is analytic  on $\Re s < 0
 $. Then $P_f^\ast (s)$  has a pole at $s = -1$  of  order $2$, and,  together with the tameness  of $\Lambda(s)$ at $s = 1$,   this entails the following  singular expressions  for $P_f^\ast (s)$ and  $P_r^\ast(s)$ at $s = -1$, 
 $$ P_f^\ast (s) \asymp \frac 1 {(s+ 1)^{2}}, \qquad P_r^\ast(s) \asymp \frac {1}{h({\cal S})} \frac 1 {(s+ 1)^{3}}.$$
The tamenesses of $P_f^\ast (s)$ and $\Lambda(s)$  at $s = 1$ are enough to deduce, using standard Mellin inverse transform \cite{FlGoDu}, the estimates, for $z\to \infty$, 
\begin{equation} \label{GZ}
P_f(z) = z \log z  \, (1+o(1)),\qquad 
P_r(z) = \frac{1}{2h(\mathcal{S})} z \log^2 z\,  (1+o(1))\, .\end{equation}

Now, we wish to return to the Bernoulli model, with Depoissonization techniques; we deal with Theorem \ref{thm1} and  first prove that $P_r(z)$ satisfies Assertions $(a)$ and $(b)$. 

 Assertion $(a)$ is easy to deduce from (\ref{GZ}) in some cone ${\cal C} (\theta_1)$. For Assertion $(b)$, we  study $P_f(z)$ and observe that 
 $P_f(z)$ can be written as $ P_f(z)  =  z^2 e^{-z} G(z)$, where
 \begin{equation} \label{Fz2}
  G(z) = \sum_{k = 0} \frac {z^k} {k!} g(k) 
 \quad\hbox{with }\quad g (k):= \frac 1 {k +1} {\log (k +2)}.
 \end{equation}
  involves  the sequence $g =  T^2[f]$.  
  As the sequence  $g$ admits an analytical continuation to the half plane $\Re (z) >0$, we apply  Theorem \ref{thm2}  [$(ii)\Longrightarrow  (i)$].  Then,  
 for some $\theta_2$, and for all
linear cones ${\cal C}(\theta)$ with $\theta <\theta_2$, there exist $\delta <1$ and $A >0$ such that the exponential generating function $G(z)$ of $g$ satisfies 
\begin{equation} \label{cone}
  z \not \in {\cal C}(\theta) \Longrightarrow  \lvert G(z) \rvert \le A \exp (\delta \lvert z \rvert).
  \end{equation} 
This exponential bound is then transfered to $P_r$, as we now explain. First, in accordance with (\ref{SK2}) and (\ref{Fz2}), the equality holds, 
\[
P_r(z) e^z = e^z \sum_{w \in \Sigma^\star} P_f(zp_w) = z^2\, \sum_{w \in \Sigma^\star} p_w^2\,  G(p_w z) \, \exp (z- p_wz).
\]
 For some $\gamma \in ]0, 1[$, we  consider 
 the cone $\hat {\cal C}(\gamma)$ defined  in \eqref{cones},  with  $\gamma$ large enough to ensure the inclusions   $\hat {\cal C}(\gamma) \subset {\cal C}(\theta_1)$ (with $\theta_1$ relative to Assertion $(a)$ for $P_r(z)$) and $\hat {\cal C}(\gamma) \subset {\cal C}(\theta_2)$  (with  $\theta_2$ relative to Eqn \eqref{cone} for $G(z)$).  
 When $z$ does not belong to $\hat {\cal C}(\gamma)$, it is the same for all the complex numbers $p_w z$, and, with Eqn \eqref{cone}, each term of the previous sum satisfies the inequality,\[
 \lvert G(p_w z) \exp (z- p_wz) \rvert \le A \exp \left[\delta p_w \lvert z \rvert + \Re (z) (1-p_w)\right] 
 \le A \exp [\lvert z \rvert (\delta p_w + \gamma (1-p_w) )],  \]
 and,  with $\alpha := \max (\delta, \gamma)$,  \quad $ \lvert G(p_w z) \exp (z- p_wz) \rvert  \le A \exp (\alpha \lvert z \rvert). $

\smallskip 
Finally,  we have shown: 
\[
z \not \in \hat {\cal C}(\gamma) \Longrightarrow |P_r(z) e^z| \le B |z|^2 \exp (\alpha \lvert z \rvert)\qquad \hbox{with} \quad B:= A \Lambda(2), \quad \alpha := \max(\gamma, \delta).
\]
 Now, for $\lvert z \rvert$ large enough, and $z \not \in \hat {\cal C}({ \alpha})$, we obtain $\lvert P_r(z) e^z \rvert \le C \exp (\alpha' \lvert z \rvert)$ with $\alpha' \in ]\alpha, 1[$ and a given constant $C$. Finally, Assertion $(i)$ of Theorem \ref{thm2} holds. 
Applying Theorem \ref{thm1} to $P_r(z)$ entails the estimate $r(n) \sim P_r(n)$ and ends the proof.

\section{The  Rice-Mellin path  } \label{Rice}

  The Rice-Mellin path  deals with the Poisson sequence $\Pi[f]$ and performs three steps.    
  \smallskip 
  \begin{itemize} 
  \item [$(1)$]  It    proves the existence of an {analytical lifting} $\psi$  of the sequence  $\Pi[f]$,  on a halfplane $\Re s >c$ (for some $c$). It uses
  the (direct) {Mellin} transform and   the {Newton} interpolation, 
  {without any other} condition on the sequence $f$.
  
  \item  [$(2)$]   If  moreover $\psi$ is of {\em polynomial growth} ``on the right'',  the binomial relation  \eqref{binom} is transfered into a {Rice} integral  expression

  \item   [$(3)$]   If moreover $\psi$ is {\em tame} ``on the left'', the integral is {\em shifted} to the left; 
    this 
   provides  the asymptotics of  of the sequence $f$. 
   
    \end{itemize}

 \smallskip   
 The main results are due to  Norl\"und \cite{No1,No2}, then to Rice who  popularized them. Later on, with the paper \cite{FlSe}, 
   Flajolet and Sedgewick  brought this methodology into the AofA domain.   The Rice-Mellin method is also well described in \cite{Fl0}. There exist  many  analyses of various  data structures  or algorithms that are based on  the application of the  method:  tries (\cite {FlSe1, FlReSz, ClFlVa, CeVa}), digital trees (\cite{FlSe1, HuVa}), or fine complexity analyses of sorting or searching algorithms on sources (\cite{ClNTVa, ClFiNTVa}).  

\smallskip   
     The situation for applying the Rice method is  not the same as in  Section \ref{DePoi}: previously,  due to Theorem \ref{thm2}, we know exactly when the Depoissonisation method may be applied.   This is not the case  for the Rice method. The  litterature well   explains how to use this method in  various cases of interest. But,  the main question, analogous to 
     Theorem \ref{thm2}:  
     
     \begin{quote} 
     {\em What are  sufficient conditions  on the sequence $f$  that would entail {polynomial growth} and tameness of $\psi$?}
     \end{quote} 
      is never asked.  
  This is the main object of the paper, which obtains  a first (original) result in this direction:   
  
   \begin{theorem}\label{thmfdb}  Consider a pair $(d, b)$ made with a real $d$ and an integer $b \ge 0$.   The Rice method can be applied  to any {\em basic} sequence $F_{d, b}$  defined as
   \begin{equation} \label{fdb}
    F_{d, b}(n)  = n^d \, \log ^b  n, \quad  n \ge 2\, .
    \end{equation}
\end{theorem}

 This will be proven in Section \ref{RiceLap}.  Now, the present section  describes the general framework of the Rice-Mellin method and its  three steps, as previously stated (in Sections 3.1,3.2, and 3.3). Then Section \ref{tame?} asks the main question: When does the Rice-Mellin method may be applied? This introduces the next Section \ref{RiceLap} which provides a partial answer, via Theorem \ref{thmfdb}.

  \subsection{The Rice path :  Analytic  lifting  $\psi$  of $\Pi[f]$  }
 
\begin {proposition} {\rm [Nordl\"und-Rice]}  The sequence $\Pi[f]$  associated with a reduced\footnote{The reduced notion is defined in Definition \ref{red}} sequence $f$ of degree $c<-1$ admits  as  an analytic lifting  the function $ \psi$,   
 \begin{equation} \label{psiNewton}  \psi(s) = \sum_{k\ge 0}   (-1)^k  \frac{f(k)}{k!}   {s(s-1) \dots (s-k+1)}\, ,  \qquad (\Re s >c)\, ,
\end{equation}
which is also an analytic extension of $P_f^\ast(-s)/\Gamma(-s)$.
\end{proposition}

\begin{proof}
 In the strip ${\cal S}(0,   -c)$,  
 the Mellin transform  $P_f^\ast (s)$ of $P_f(z)$ exists   and satisfies
\begin{align*} 
\frac {P_f^\ast (s)}{\Gamma(s)} &=\frac 1 {\Gamma(s)} \sum_{k\ge 0 } \frac { f(k)} {k!} \int_0 ^\infty e^{-z} z^{k} z^{s-1} dz
= \sum_{k\ge 0} \frac { f(k)}{k!}\,  \frac { \Gamma(k +s) }{\Gamma(s)}\, 
\end{align*}
where the exchange  of  integration and summation is justified by the estimates given in \eqref{Stirling}. 
On the strip  $ {\cal S}(c , 0)$,  the  series is  a Newton interpolation series, 
\begin{equation} \label{Newton}
  \psi(s) :=  
{\frac {P^\ast(-s)} {\Gamma(-s)} =  
 \sum_{k \ge 0} (-1)^k \frac { f(k)}{k!} {s(s-1) \dots (s-k+1)}\, , } 
 \end{equation}
which {converges in right halfplanes} and thus on $\Re s >c$. 
 Moreover,  Relation \eqref{Newton}, together the binomial relation \eqref{binom}, entails   the equality
 $$\psi(n) =    \sum_{k =  0}^n (-1)^k   \frac {f(k)}{k!}   {n(n-1) \dots (n-k+1)}= \sum_{k =  0}^n (-1)^k {n \choose k}  f(k) = \Pi[f](n)\, .$$
 This  proves that  $\psi$ provides   an  analytic lifting  of  the sequence $ \Pi[f]$ on $\Re s >c$
   which is also  an  analytic extension of $P_f^\ast(-s) /\Gamma(-s)$. 
   \end{proof}

 \subsection{Shifting to the right.  Rice transform. }
The binomial relation between $f(n)$ and $ \psi(n) = \Pi[f](n)$ is  now transfered into a Rice integral. 
\begin{proposition} \label{right}  {Assume that  the analytic lifting $\psi$  of $\Pi[f]$ is of {polynomial growth}  on  the halfplane $\Re s >c$, with $c<-1$. Then, for any  $ a \in ]c, 0[$ and $n \ge n_0$,  the sequence $f(n)$ admits an integral  representation of the form
$$ 
  f(n) =  \sum_{k = 0}^n (-1)^k {n \choose k} p(k) \Longrightarrow   f(n) = -
  \frac 1 {2i\pi} \int_{a-i\infty}^{a+i\infty}   L_n(s)\cdot \psi(s) \,ds  $$
{with the Rice kernel } \quad 
$$ {  L_n(s)=   \frac{ (-1)^{n+1} \, n!\, 
 }{s(s-1)(s-2) \dots (s-n)} =  \frac {\Gamma(n+1) \Gamma(-s)}{ \Gamma(n+1-s)}= B(n+1, -s)}\, . 
$$}
\end{proposition}

\begin{proof} 
(Sketch) Use the  Residue Theorem and  the   polynomial growth of  $\psi(s)$ ``on the right''.
First, we  consider  the   rectangle ${\cal A}_M$  delimited by the  contour $\tau_M$ defined  by the  two vertical lines $ \Re s  = a$  and  
$\Re s = n+M$ and  the two horizontal lines $\Im s = \pm M$. If the contour $\tau_M$  is taken  counterclockwise, then the Residue Theorem  applies and entails the equality 
\begin{equation} 
\begin{split}
\frac{1}{2i\pi }\int_{\tau_M} L_n(s)\cdot \Pi[f](s) ds &=  \sum_{k= 0}^n {\rm Res} [L_n(s) \cdot \Pi[f](s) ;s=k] \\
&=(-1)^n\sum_{k= k_0}^n (-1)^k {n \choose k} \Pi[f](k) = f(n)\\
\end{split}
\end{equation}
Next, the integral on  the curve $\tau_M$ is the sum of four integrals. Let  now $M$ tend  to $\infty$.  The integrals on the right, top and bottom lines tend to $0$, due to the polynomial growth of the function $\Pi[f](s)$.   
The integral on the left  becomes 
$$  
 -\int_{d-i\infty}^{d+i\infty} L_n(s)\cdot \Pi[f](s) ds, $$  and we have proven   the result. For  details on the proof,  we may refer to papers  \cite{No1,No2} or \cite{FlSe}.
\end{proof}

 This integral representation is valid for  any abscissa $a$ which belongs to the interval  $]c, 0[$.
 We now  shift  the vertical line $\Re s = a$ {to the left}, and 
 thus  use  {\em tameness conditions} on $\psi$ at $
 s = c$, as defined in Section \ref{def}. 

\subsection{Tameness  of $\psi$ and shifting to the left.  }

\medskip 

\begin{proposition}  Consider   a sequence $f: n \mapsto f(n)$  with ${\tt val} (f) = 0$ and ${\tt deg}(f)= c<-1$. If  the lifting $\psi$ of  $\Pi[f] $ is  {tame} at $s = c$ with a region ${\cal R}$ of tameness, then  
{$$ {f(n)} =  - \left[{ \sum_{k \mid s_k \in {\cal R} } {\rm Res} \left[L_n(s) \cdot \psi(s) ;  s = s_k\right] }+  \frac 1 {2i\pi}   \int_{{\cal C}} L_n(s) \cdot \psi \, ds  \right],\, 
$$
where the sum is  taken over the poles $s_k$ of $\psi$    inside ${\cal R} $.}
\end{proposition}

\begin{proof} (Sketch)
 The proof is similar to  the  previous proof. 
  With the tameness of  $\psi(s)$  at $s = c$, 
   consider the region  ${\cal R}_M$ which is the intersection of the domain ${\cal R}$ with the horizontal strip $ |\Im s| \le M$, and denote   ${\cal C}_M$ the curve  (taken  counterclockwise) which  borders the region ${\cal R}_M$. As
    $\Pi[f] (s)$ is meromorphic  in ${\cal R}_M$,  we apply the Residue Theorem to the function $L_n(s)  \,  \Pi[f] (s)$  inside ${\cal R}_M$,  and obtain $$  \frac 1 {2i \pi} \int_{{\cal C}_M} L_n(s) \cdot  \Pi[f] (s) \, ds =  \sum_{s_k \in {\cal R}_M} {\rm Res}\left[L_n(s) \cdot \Pi(f](s);  s = s_k\right]$$
where the sum is extended to all poles $s_k$ of $\Pi[f](s)$  inside  the domain  ${\cal R}_M$. 
Now, when $M\to \infty$, the integrals on the two horizontal segments tend to 0, since $\psi(s)$ is of polynomial growth, and 
\begin{align*}
\lim_{M \to \infty}  \int_{{\cal C}_M} L_n(s) \cdot \Pi[f](s)\, ds & =  \int_{d-i\infty}^{d+i\infty}  L_n(s) \cdot  \Pi[f](s) \,ds - \int_{{\cal C}} L_n(s) \cdot \Pi[f](s)  ds \\ &=  2 i \pi  \sum_{s_k \in {\cal D}} {\rm Res}\left[L_n(s) \cdot \Pi[f] (s);  s = s_k\right]\, .
\end{align*}
\end{proof}

\subsection{Sufficient conditions for tameness  of $\psi$?} \label{tame?}
 The path  is often  easy to apply,    as soon one has some knowledge of  
$$   \psi (s) : =  \frac 1 {\Gamma(-s)} {P_f^\ast(-s)} $$
 on  the left of the vertical line $\Re s = c$. 
As $\psi(s)$ is closely related to the {Mellin transform} $P_f^\ast(-s)$, 
 {meromorphy}  is  often easy to prove,  and the poles   often easy to find.    In many  natural contexts,   the polynomial growth  and the tameness of the Mellin transform $P_f^\ast (s)$  generally hold, and are often used in the Depoissonisation approach [see  Section \ref{triedepo}]. But the main difference between the Rice method and the Depoissonisation method is the division  by $\Gamma(s)$.  
 
\smallskip
Sometimes, and this  is often the case in classical  tries problems, the factor $ \Gamma(s)$  already  appears in $ P_f^\ast (s)$, and $\psi(s)$ has an explicit form, from which its polynomial growth may be easily proven.  For instance,   for   the toll $f = f_1$ associated to the path length, then  $P_f^\ast(s)= \Gamma(s+1)$ and
$ \psi(-s)$  is explicit,  and equal to $ s+1$.   This is also the case for polynomial  tolls $f$ of the form $f= T^{-m}  [f_1] $ with $m \ge 1$.  

\medskip
But  what about   other sequences, for instance    
the  ``generalized sorting'' toll $ f(k) = k \log^b k$, with an integer $b \ge 1$, or the basic sequence $f(k) = k^d \log^b k$ (with $d \in {\mathbb R} $ and an integer $b \ge 1$). The following expansion holds, that involves the $b$-th derivative of the Riemann $\zeta$ function and  generalizes  \eqref{Fast}, 
\begin{equation} \label{zeta}
P_f^\ast(-s) = 
(-1) ^b\zeta^{(b)} (s- (d-1)) + H_1(s) \, , 
\end{equation}
 where $H_1(s)$ is analytic  on $\Re s > d-1$. Then  $ P_f^\ast(-s)$ has a pole of order $b+1$ at $s = d$, and  principles of Depoissonisation apply in this case,   due to good properties of the Riemann function.   Now, in the Rice method,  the function $\psi$ satisfies $\psi(s) =  P_f^\ast(-s)/\Gamma(-s)$ and it  has a pole of order $b$ at $s = d$. 
  However, the function $1/\Gamma(-s)$, even though it is analytic on the half-plane $\Re (s) > 0$, 
  is  of exponential growth along vertical lines. The Stirling  formula   indeed  entails the estimate
 $$ \frac 1 {\Gamma(x+iy)} = \frac 1 {\sqrt{ 2 \pi}}\,  e^{\pi |y|/2} \,  |y|^{1/2-x}, \qquad \hbox{as $|y| \to \infty$}\, .$$
 It is  thus not clear whether $\psi (s)$ is tame at $s = 1$. Then, the Rice method seems to  have a more restrictive use than the Depoissonisation method. 
As we wish to compare the power of the two  paths [Depoissonisation path and Rice path], we ask the two (complementary) questions: 
\begin{quote}  
{\em 
    Is  the Rice path  only useful for very specific tolls, where the  Mellin transform $P_f^\ast (s)$ of the Poisson transform $P_f(s)$   factorizes with the factor $\Gamma (s)$, or is it  {useful}  for  more general tolls? }
  \end{quote}
 
 This leads us to study sufficient conditions under which  the analytic lifting $\psi$ may be proven to be tame. We now   propose to use the (inverse) Laplace  transform. With this tool, we prove the tameness of $\psi$  for basic sequences (see Theorem \ref{tame}) which itself leads to Theorem \ref{thmfdb}.

\section{The Rice--Laplace approach.} \label{RiceLap}
As previously,  we deal with  the pair $(P_f, p= \Pi[f])$  made with the Poisson transform and the Poisson sequence
$$ {P(z)} = e^{-z} \sum_{n \ge 0} { f(n)}\,  \frac{z^n}{n!} =  \sum_{n \ge 0}  (-1)^n {p(n)} \, \frac{z^n}{n!}\, , $$
 together with  the involution $\Pi$ between the sequences  $f$ and $p := \Pi[f]$ which will play in important role  here. 
 
 \smallskip
 This Section is devoted to the proof of  Theorem~\ref{tame}  which itself entails  Theorem~\ref{thmfdb}. We first introduce  in Section \ref{secInvLap} a new expression of the analytical  extension $\psi$ which deals with the inverse Laplace transform $\hat \varphi$ of the extension $\varphi$ of the canonical sequence attached to a sequence $F$. Then, the  sequel of the  present section focuses on basic sequences 
 $F_{d, b}$. We first obtain in Section \ref{secfdb} a precise expression of the  inverse Laplace transform $\hat \varphi$  of the 
 extension $\varphi$ of the canonical sequence of $F_{d, b}$. We then use in Section \ref{sectamefdb} the expressions  of  the two previous Sections and prove the tameness of the $\psi$ function associated to a basic  sequence.  This leads to the main theorem (Theorem~\ref{tame}) of this Section. 
 
 \subsection {A new expression for $\psi$ with  the inverse Laplace  transform.} \label{secInvLap}
  The expression of $\psi$ given in \eqref{psiNewton} is not  well-adapted to study its tameness.  Under the existence of an analytic lifting of the sequence  $f$, of polynomial growth on {\em halfplanes}, we obtain another expression of $\psi$ which is easier to deal with.

\smallskip  
  We first recall  the principles of Section \ref{VD} : with a initial sequence $F$ of degree $d$,  we associate its canonical sequence $f$, and take it as our new sequence $f$.  If  the old $F$ admits an analytic lifting   of polynomial growth  on $\Re s >0$, then  the new  sequence $f$ admits an analytic lifting  $\varphi$ on $\Re s >-1$  that satisfies $\varphi(s) =   O(|s+1|^c)$ there, with 
  $c<-1$.    We now deal with this new sequence $f$, and then return (later) to the initial sequence $F$.   
 
 \begin  
{proposition}   \label{LapInv} Consider    a sequence $f$ which    admits an {analytic} lifting  $\varphi$    on $\Re s>-1$,     with the estimate  $ \varphi(s)  = O(|s+1|^c)$ with $ c<-1$.    Then:

    \begin{itemize} 

\item[$(i)$]    The function $\varphi$ admits an inverse Laplace transform 
$\hat  \varphi$ whose restriction to the real line $[0, + \infty[$  is written as  the Bromwich integral for $b \in ]-1, 0[$, 
$${ \hat \varphi(u) = \frac 1 {2i\pi} \int_{\Re s = b} \varphi(s) \, e^{su} ds} \, ,\qquad   \hbox{  and   satisfies  $ |\hat \varphi(u)| \le K  e^{bu}$}\, .$$

\item[$(ii)$]
There exists an  analytical  lifting $\psi$ of the sequence $\Pi[f]$  on   $\Re s >-1$,  that  is expressed  as an integral on the real line, 
\begin{equation} \label{psilapinv}
 \psi(s) = \int_0^{+ \infty}  \hat \varphi(u)  \cdot ( 1 -e^{-u})^{s} du \, .
\end{equation}
\end{itemize}
\end{proposition}

 \begin{proof} 
 
 $(i)$ In a general context, where the   analytic lifting $\varphi(s)$ is  only defined  on $\Re s >0$, the Bromwich integral  is written as 
 $${ \hat \varphi(u) = \frac 1 {2i\pi} \int_{\Re s = b} \varphi(s) \, e^{su} ds}, \qquad \hbox{ (with $b>0$)}\, .$$
Here, the hypotheses on $\varphi$ are stronger: we can shift the integral on the left and choose $b \in ]-1, 0[$. 
Moreover, the  Bromwich integral is normally convergent, and the exponential   bound  on $\hat \varphi(u)$ holds.

\smallskip
$(ii)$ We  use the involutive character of $\Pi$ and apply Proposition  \ref{right} to the pair $(p := \Pi[f] , f= \Pi^2 [f])$. In the classical Rice-Mellin path, it is applied to  the pair $(f, \Pi[f])$, when  $\Pi[f]$ is of polynomial growth, and it transfers the binomial  expression of $f$  in terms  of $\Pi[f] = f$ into an integral expression. Here,  due to the polynomial growth of $f= \Pi^2[f]$ on $\Re s >-1$,  it transfers the binomial  expression of $\Pi[f]$  in terms  of $\Pi^2[f] = f$ into a Rice integral,  with $b \in ]-1, 0[$, 
$$ {p(n) = \frac 1 {2i \pi}\int_{\Re s = b} \varphi(s) L_n(s) ds}, \qquad L_n(s) =  \frac {\Gamma(n+1) \Gamma(-s)}{\Gamma(n+1-s)}\, .$$
We now deal with  the Beta function  
$$ B(t+1, -s) =  \frac {\Gamma(t+1) \Gamma(-s)}{\Gamma(t+1-s)} \, ,$$
that is well defined for  $\Re t >-1$ and $\Re s<0$,  and admits an integral expression 
$$ B(t+1, -s) =\int_0^\infty  e^{su} (1-e^{-u} )^t du, \qquad  \hbox {(for  $\Re t >-1,  \, \Re s<0$)} \, .$$
 Together with the equality $L_n(s)  = B(n+1, -s)$,  
 this entails an analytic extension   $\psi$ of the sequence $\Pi[f]$  on the halfplane $\Re t >-1 $, 
$$\psi(t)  = \frac 1 {2i \pi}\int_{\Re s = b} \varphi(s) B(t+1, -s)  ds, \quad (b<0) $$
with an integral expression, 
$$ \psi(t)  = \frac 1 {2i \pi}\int_{\Re s = b} \varphi(s) \left[ \int_0^\infty  e^{su} (1-e^{-u} )^t du\right]  ds \, .$$
With properties of $\varphi$,  it is possible to exchange the integrals: then,  the equality holds
$$ \psi(t) = 
\int_0^\infty  (1-e^{-u})^t  \left[  \frac 1 {2i \pi}\int_{\Re s = b} \varphi(s)\,  e^{su} ds \right] du, $$
and the second integral   is  the inverse Laplace transform $\hat \varphi$ of $\varphi$. This ends the proof.
\end{proof}

\medskip
  With any  sequence   $F$ of polynomial growth, with an analytical lifting, we associate   its canonical  sequence $f$ 
  which fulfills hypotheses of Proposition  \ref{LapInv}, with an analytical lifting $\varphi$.    Then,  Proposition \ref{LapInv}  provides 
  an  integral form    for an  analytical lifting $\psi$ of the $\Pi[f]$ sequence, in tems of   an integral   over the real line $[0, +\infty[$ (described in \eqref{psilapinv}) which involves the (inverse) Laplace  transform $\hat \varphi$ of $\varphi$.  
 
   \medskip
  Then, there are  two steps which deal with  a sequence $f$, and aim at studying   the tameness of the analytical extension $\psi$ of $\Pi[f]$: 
  
\smallskip
$(i)$  we  transfer properties of   $\varphi$ into  properties of  its  inverse Laplace  transform $\hat \varphi$; 

\smallskip
$(ii)$   we  use properties of $\hat \varphi$ and   
 study  the tameness of $\psi$, via  the  integral representation  \eqref{psilapinv}.

\smallskip 
 We now perform these two steps.   The  (inverse) Laplace  transform is not well studied, and we do not know   a general  result which describes how properties of   a general  function $\varphi$  are transfered into properties of its transform $\hat \varphi$.  This is why we do not perform  this study  for a general sequence $f$ and restrict ourselves to  the case of basic sequences.
 
 \subsection{Canonical sequences associated to basic sequences and their inverse Laplace transforms.}   \label{secfdb} We then deal with  (initial) basic sequences defined as  follows: 

\begin{definition}
 Consider a pair  $(d, b)$   with  a real $d$ and an {\em integer} $b \ge 0$.   
 A sequence $F$ is called  basic with pair  $(d, b)$   if it  satisfies
 \begin{equation} \label{basic}
F(k) = k^d\,  \log^b k  \quad \hbox{ for any $k \ge 2$} \, . 
\end{equation} 
\end{definition} 

 For applying  Proposition \ref{LapInv}, we have to deal with canonical sequences associated to basic sequences.   
  Letting   $\ell:= \sigma (d)$ with $\sigma$ defined in \eqref{ell}, the   canonical sequence  $f$ associated with $F$  can be extended to a function $f$ defined on   $]-1, +\infty[$ as 
$$f(x) = \log^b  (x+ \ell)   \frac { (x+\ell)^d}{(x+1) (x+2)  \ldots (x+ \ell)}  =   \log^b  (x+ \ell) \, (x+\ell)^{d-\ell} \, U\left( \frac 1 {x+\ell}\right)\, , 
$$
  and involves  a function $U$ defined as   $U(u) = 1$ for $ d <0$ and, for $d \ge 0$ as  
\begin{equation} \label {Ud}U(u) = (1-u)^{-1} (1-2u)^{-1} \ldots (1- (\ell -1) u) ^{-1} \quad   \hbox{(with $ \ell = 2 + \lfloor d\rfloor$)} \, .
\end{equation}   Then, for $d \ge 0$, the coefficient $a_j := [u^i]   U(u)$ satisfies 
$ a_j = \Theta  (\ell-1)^j $. 
Finally,  we have proven:

 \begin{lemma} 
 Consider  the basic sequence $F_{d, b}$  with pair $(d, b)$ and let $\ell:= \sigma(d)$ with $\sigma$ defined in \eqref{ell}. The   canonical  sequence  $f$ associated to  $F_{d, b}$  is extended in a function $\varphi$ defined on $\Re s >-1$ as 
 $$ \varphi(s)  =  (s+\ell)^{d-\ell}   \log ^b (s+\ell) \,    U\left( \frac 1 {s+\ell}\right)   $$
 where $ U$  is defined in \eqref{Ud}. For $d \ge 0$, the coefficient $a_j := [u^i]   U(u)$ satisfies 
$ a_j = \Theta_d  (\ell-1)^j $. \end{lemma}

  The following lemma describes the  inverse Laplace transform $\hat \varphi$. 
 
\begin{proposition} \label{cano2} Consider a basic sequence $F$ with pair $(d, b)$. Then
the inverse Laplace transform $ \hat \varphi (u) $   is a linear   combination of functions,    for $m \in [0..b]$, 
\begin{equation} \label{hatvarphidb}
 e^{-\ell u} \, u^{-c-1} \, \log^m u    \, \left [1 +  u V^{\langle m \rangle}(u)\right], \qquad \hbox {with} \quad  |V^{\langle m \rangle} (u) |\le A_{(d, b)}\,   u\,  e^{(\ell-1) u}\, .
 \end{equation} 
\end{proposition}

\begin{proof}
There are three main steps, according to the type of the basic sequence. 
 
\medskip
{\em Step 1.}  We begin with   the  particular case when $\varphi(s)$ is  of the form  $\varphi(s) = 
 (s+\ell)^{c}$ (with $c <-1$). Its inverse Laplace transform $\hat \varphi$ is   then 
 $$ \hat  \varphi (u) =  \frac 1 {\Gamma (-c)} e^{-\ell u} \,  {u^{-c-1}} \, .   $$
 
  \medskip
 {\em Step 2.} 
 We  now consider a function   (without logarithmic factor) of the form  
 \begin{equation}  \label{varphic}
   \varphi(s) = \varphi_c(s) =  (s+ \ell)^{c}  U \left( \frac 1 {s+\ell} \right) = \sum_{j \ge 0} a_j \, (s+\ell)^{c-j}\, .
   \end{equation}
  Then $\varphi$ is a linear combination of functions  of  Step 1 and  the  inverse Laplace transform $\hat \varphi$ of $\varphi$  
  is  written as
\begin{equation} \label {hatvarphic}
    \hat \varphi_c(u) = e^{-\ell u}  \ \frac { u^{-c-1} }{\Gamma(-c)}[ 1+ V_c(u) ] , \quad \hbox{with}\ \   V_c(u) := \sum_{j \ge 1}  a_{j} {u^{j}} G_j(c)\, ,
    \end{equation} 
     where  the function $G_j$  is the rational fraction which associates with $c$ the ratio 
    \begin{equation} \label{Gj}
    G_j(c) :=  \frac {\Gamma(-c)}{\Gamma (j-c)} = \frac 1 {-c(1-c)\ldots (j-1-c)} \, .
    \end{equation} 
    As $c<-1$,  the inequality $G_j(c) \le (1/j!) $ holds  and this entails  the inequality $|V_c(u)| \le A \,   u\, e^{(\ell-1)u}$, where the constant $A$ only depends on $d$. 
  
 \medskip
 {\em Step 3.}    We add finally a  logarithmic factor and consider  a function  of the form  
  \begin{equation} \label {g3}
   \varphi(s)= (s+\ell)^{c} \,   \log^b (s+ \ell)\, U \left( \frac 1 {s+\ell} \right)
   \end{equation}
     which is   written as a  $b$-th derivative. Indeed,   
     the equality holds
   $$  U \left( \frac 1 {s+\ell} \right)  (s+\ell)^{c} \log^b (s+\ell)  =   \frac { \partial^b}{ \partial t^b}  \left[(s+\ell)^{c +t}  U \left( \frac 1 {s+\ell} \right) \right] \bigg|_{t = 0} \, , $$
and we can take the derivative ``under the Laplace integral'':   we  then deduce that the  inverse Laplace transform  $\hat \varphi$   of the function $\varphi$ defined in \eqref{g3}  is equal to 
   $$    \frac { \partial^b}{ \partial t^b}  \hat \varphi_{c+t} (u)\bigg|_{t = 0} =  e^{-\ell u}  \frac { \partial^b}{ \partial c^b}  \left[ \frac { u^{-c-1} }{\Gamma(-c)}( 1+ V_c(u) ) \right]  \, .$$ 
 The coefficient  of $u^j$ in the  $k$-th derivative of $c \mapsto V_c(u)$  involves  the  $k$-th derivative of the function $ c \mapsto G_j(c)$, defined in \eqref{Gj} which satisfies  the inequality 
 
 \centerline { $|G^{(k)}_j(c)| \le  A_k \log^k (j+c) \,  G_j(c)$ \ \ for some constant $A_k$.} 
 \vskip 0.1 cm Then, the inequality holds, 
   $$  \left| \frac { \partial^k}{ \partial c^k}  V_c(u) \right|\le   A_{(d, b)} \,   u \, e^{(\ell-1)u}  \, , $$
    and involves  a constant $A_{(d, b)}$ which depends on the pair $(d, b)$.   On the other  hand, the  following $m$-th derivative  is a linear combination of the form 
   $$  \frac { \partial^m}{ \partial c^m} \left[  \frac{u^{-c-1}}{\Gamma(-c)} \right]  =  u^{-c-1}  \left[ (-1) ^m  \sum_{a = 0}^m  {m  \choose a}\, 
     (\log^{a} u) \,   H^{(m-a)}(c)\right]  \, ,   $$
     where $H$ is the function  defined as $H(c) = 1/\Gamma(-c)$.  This ends the proof. 
  \end{proof}

 \subsection {Tameness of $\psi$   in the case of canonical versions of basic sequences. } \label{sectamefdb}

The expression \eqref{psilapinv} leads us  to   the operator ${\cal I}_s$ defined as 
  \begin{equation} \label{Is}
    {\cal I}_s [h] := \int_0^\infty   h(u) (1-e^{-u})^s  du  
  =  \int_0^\infty   h(u)  \, u^s\,  \left( \frac {1-e^{-u}}{u}\right)^s  du \, , 
  \end{equation}
 and the  relation $\psi(s) = {\cal I}_s [ \hat \varphi]$ holds.  Using the expressions of $\hat \varphi$ obtained in Proposition \ref{cano2} together with the bound on the remainder term  $ |V^{\langle m \rangle} $ entail  the decomposition:   
 \begin{lemma} 
   Consider a basic sequence $F$ with pair $(d, b)$. Then
the analytic extension $\psi$   of the sequence $\Pi[f]$ associated with the canonical extension $f$ of $F$    is a linear   combination of functions,    for $m \in [0..b]$, each of them being the sum of a main term $A^{\langle m \rangle}(s)$ and a remainder term $O(B^{\langle m \rangle}(s))$, with
\begin{equation} \label{psidb}
A^{\langle m \rangle}(s) =   {\cal I}_s \left[ e^{-\ell u} \, u^{-c-1} \, \log^m u \right], \qquad  B^{\langle m \rangle}(s)  = 
{\cal I}_s \left[ e^{- u} \, u^{-c} \, |\log^m u | \right]
 \end{equation}
 \end{lemma} 

We then introduce the functions  
\begin{equation} \label{Fs}
 N_s (u) := \left( \frac {1- e^{-u}}{u}\right)^s, \    M_s(u) :=  \left[ \left( \frac {1- e^{-u}}{u}\right)^s  -1 \right]  \, , 
   \end{equation}
defined on $[0, + \infty]$,   that satisfies the two estimates, with $\sigma:= \Re s$
  \begin{equation}  \label {estNs}
 N_s(u) = \Theta(1),\quad   (u \to 0) ,  \qquad    \left| N_s(u)  \right|  = O(u^{-\sigma}) \quad (u \to \infty)  
 \end{equation} 
 \begin{equation}   \label {estMs}  M_s (u)   = \Theta(u) \quad   (u \to 0) , \qquad    \left| M_s(u)  \right|  = O(u^{-\sigma}) \quad (u \to \infty, \sigma>0 )  \, . 
    \end{equation}
  With the estimates  \eqref{estMs}, and for ``good'' functions $h$,   the integral ${\cal I}_s[h]$ may be compared to the Mellin transform  $h^\star (s+1)$.

    We now apply this idea  to the  particular cases where  $ h= \hat \varphi (u)$ arises in connection with basic sequences $F$.   The main term  $A^{\langle m \rangle}(s)$   is  then compared to  the twisted version of the  $\Gamma$ function and its $m$-th derivative, that are  defined for $\Re s >0$,   integers  $m \ge 0$, and  $\ell \ge 1$ as 
\begin{equation} \label{twGamma}
 \Gamma_\ell ^{(m)} (s) :=
 \int_0^\infty e^{-\ell u} u^{s-1} \log^m u \,   du \, .
 \end{equation}
We  indeed prove the following:  

\begin{lemma}
For $\Re s \ge 0$,  the two functions $A^{\langle m \rangle}(s)$ and $B^{\langle m \rangle}(s)$  are bounded on  the halfplane $\Re s \ge 0$. 
For any integer $m\ge 0$ and any  integer $\ell \ge1$, the  two functions
$$ 
 A^{\langle m \rangle}(s) -  \Gamma_\ell^{(m)} (s-c), \quad  B^{\langle m \rangle}(s)$$
 are analytic and of bounded growth on the vertical strip  $\Re s >c -\sigma_0$, with $\sigma_0\in ]0, 1[$.
 \end{lemma}

\begin{proof} 
 $(a)$ is clear :  For $\Re s\ge 0$,  the result follows from   the inequalities  $(1-e^-u)^\sigma \le 1$,   $c<-1$,  together with  the  integrability of the function $u \mapsto e^{-\ell u} u^{-c-1} \log ^m u$ on  the interval $[0, + \infty]$.

 $(b)$ The difference  $A^{\langle m \rangle}(s) -  \Gamma_\ell^{(m)} (s-c)$ is expressed  with  $M_s(u)$, whereas   $B^{\langle m \rangle}(s)$ is expressed with $N_s$, both defined in \eqref{Fs}.  Together with their  estimates~\eqref{estMs}, \eqref{estNs}, this   leads to the following bounds, for any $\rho>0$, 
 $$   A^{\langle m \rangle}(s) -  \Gamma_\ell^{(m)} (s-c) = O_\rho \Big(  \Gamma_\ell(\sigma -c +1 - \rho) \Big), \ \     
  B^{\langle m \rangle}(s) =  O_\rho \Big(  \Gamma(\sigma -c +1 - \rho) \Big) \, $$
  and  also to the analyticity of the functions of interest on the vertical strip  $\Re s >c -\sigma_0$, with $\sigma_0\in ]0, 1[$. 
\end{proof}

\begin{proposition}   Consider a basic sequence $F$ with pair $(d, b)$.  Let $\ell:= \sigma(d)$ with $\sigma$ defined in \eqref{ell} and $c = d-\ell$. Then
the analytic extension $\psi$   is of polynomial growth on $\Re s > c$ and is tame at $s = c$
\end{proposition} 
\begin{proof}  The  twisted function $\Gamma_\ell^{(m)}$ is meromorphic at $s = 0$  with a pole of order  $m +1$. Moreover, with Lemma \ref{ES}, the  twisted function $\Gamma_\ell$ and its derivatives  are of exponential decrease along the vertical lines. This entails the  tameness of  $\psi$  at $s = c$. 
\end{proof}

  \subsection{Tameness of  the $\Psi$ function relative to a basic sequence.}

 Returning to the initial  sequence $F$ and the analytic extension $\Psi$ of the sequence $F$, we apply Assertions $(a)$ and $(b)$ of Lemma  \ref{shift} and we   obtain the main result of this Section,   which entails Theorem  \ref{thmfdb}.    
\begin{theorem} \label{tame}
Consider a basic sequence $F:=F_{d, b}$ 
defined as in \eqref{fdb}, denote by $\ell := \sigma(d)$ where $\sigma$ is defined in \eqref{ell}. Then, for some $\sigma_0>0$,  
the analytic continuation $\Psi(s)$ of the $\Pi[F]$  sequence  is of polynomial growth on any halfplane $\Re s \ge a > d$. Moreover, it is tame at $s = d$, and it  writes as 
\begin{equation} \label{expan}
\Psi (s) =\left[ s(s-1) \ldots (s-\ell+1)  \sum_{m = 0} ^b  a_m \,  {\Gamma_\ell^{(m)} (s-d)} \right] + B(s) \, 
\end{equation}
on the halfplane  $\Re s> d- \sigma_0$, for $\sigma_0 \in ]0, 1[$. Here, 
$B(s)$ is of polynomial growth, and   ${\Gamma_\ell^{(m)}}$  is the $m$-th derivative of the twisted $\Gamma$ function defined in  \eqref{twGamma}.
The coefficients $a_m$   are expressed with the derivatives of order $ k \le b$ of the function $s \mapsto 1/\Gamma(s) $ at $ s= \ell-d$. 
\end{theorem}

\smallskip
{\bf Remarks.}  $(a)$  With Lemma \ref{ES}, the  twisted function $\Gamma_\ell$ and its derivatives  are of exponential decrease along the vertical lines,. This entails the  tameness of  $\Psi$  at $s = d$.  \\
$(b)$ The multiplicity of the pole at $s = d$ is at most $b+1$; this is due to a possible cancellation with   the linear factor $(s-d)$ when it appears in the first part. This occurs  for the  generalized sorting toll,   ($d = 1$) and the function $\Psi$ has only a pole at $s = 1$ of order $b$.\\
$(c)$  The derivatives of the $\Gamma$ function are related to the derivatives of the  function  $\log \Gamma (s)$ which arises in  many contexts of Number Theory.  One has for instance $\Gamma'(1) = - \gamma$ (the Euler constant) and $\Gamma''(1) = \zeta(2) + \gamma^2$.  See \cite{Hab} for  other examples of computations.  \\
$(d)$ We already know the singular part  of $\Psi$ at $s = d$ which is given by the expansion \eqref{zeta}, and the  singular expansion given in \eqref{expan}  is not new. What is new is the tameness, not the singular expansion.

 \subsection{ A second proof for Theorem ~\ref{anatrie}.} This also gives an alternative proof of Theorem~\ref{anatrie} that uses the Rice-Mellin path:  We consider the singular part of the extension $\Psi$ at $s = 1$, already obtained in \eqref{zeta}, together with the  tameness  of the Dirichet series $\Lambda$ at $s = 1$, and  finally the tameness of  $\Psi$ just obtained in Theorem{tame}. This gives a proof with two or thee lines..., much more direct than the proof that we gave in Section  \ref{triedepo}.  We prefer this proof!

  \section{Conclusions. Possible Extensions.}

 \subsection{Possible extension of  Theorem \ref{thmfdb}.} We now ask the question of possible extensions of Theorem \ref{thmfdb} to more general sequences, and we  explain how it is  possible to extend our result to  sequences   $F$ what we call {\em extended basic sequences}. 

\begin{definition} \label{extfdb}  A sequence $F$ 
is  an extended basic  sequence with pair $(d, b)$ if it admits an extension  $\Phi$ on some  halplane $\Re s > a$,     which involves  an analytic  function $W$  at $0$,  with $W(0) \not = 0$, of the form 
$$ \Phi(z) = F_{d, b} (z) \,  W\left( \frac 1 z\right) .$$
\end{definition}

It is easy to extend the proof of Theorem~\ref{tame} and thus this of Theorem \ref{thmfdb}  to this more general case:  We denote by $r$ the   convergence radius of $W$, and we thus  choose  a shift $T^\ell$ with an integer which now   satisfies 
$$\ell \ge \max \Big[2 + \lfloor d\rfloor, a +1,  (1/r) +1\Big], $$
and deal with the sequence $f := T^\ell[F]$. We replace the  previous series $U$ defined in \eqref{Ud} by the series $U\cdot W$
which has now a convergence radius $\tilde r:= \min (r, 1/(\ell-1)) $ for which  the bound $1/\tilde r < \ell$ holds.  We choose $\hat r \in ]1/ \tilde r, \ell[$,  and  the new  series $V_c$ defined in \eqref{hatvarphic} satisfies $|V_c(u)| \le  A u\ e^{\hat r u}$ and  indeed gives rise to a remainder term. 

\subsection{Final comparison between the two methods.} Even with  our main result (possibly) generalized  to extended basic sequences, the Rice-Mellin path  seems to  remain  (for the moment...) of more restrictive use than the Depoissonisation path, in three aspects: 

\begin{itemize} 

\item[$(a)$]  In the Rice-Mellin path,  we need the analytic extension  $\varphi$ of $f$ to hold on a halfplane, whereas  the Depoissonisation  path needs it only on a horizontal cone.

\item[$(b)$] 
 In the Rice-Mellin path, the analytic extension $\varphi$  of $f$ involves  a precise expansion in terms of an  analytic series $W$, whereas the Depoissonisation path  only needs a rough  asymptotic estimate of $\varphi$. 

\item[$(c)$] Finally, In the Rice-Mellin path, the exponent of the $\log$ term is an integer $b$, whereas the Depoissonisation path   deals with any  real exponent.   The need of an integer exponent $b$ is related to the interpretation in terms of  $b$-derivatives, and this is a restriction which is also inherent in the method used  by Flajolet in \cite{Fl1} in a similar context. 

\end{itemize} 
These are strong restrictions... However,  most of the Depoissonisation analyses deal with extended basic sequences, where the Rice-Mellin path  may be used. We leave the final conclusion to the reader !

 \subsection {Description of a formal comparison between the two paths.}  As it is observed in  the paper \cite{HwFuZa}, there are formal manipulations which  allow us to compare the two paths.

\medskip 
In the Depoissonisation path, the asymptotics of  $f(n)$   is  manipulated  in  two  steps: first use  the Cauchy integral formula  
\begin{equation} \label{depoi-ii}
  f(n) = \frac {n!}{2 i \pi} \int_{|z] = r}  P(z)\,  e^z  \frac 1 {z^{n+1}}  dz  \, .
\end{equation}
then 
derive asymptotics of  $P(z)$ for large $|z|$ by the inverse Mellin integral
\begin{equation} \label {Me} P(z) = \frac 1 {2 i\pi}  \int_{\uparrow} P^\ast (s) z^{-s} ds = \frac 1 {2 i\pi}  \int_{\uparrow} P^\ast (-s) z^{s} ds \, , 
\end{equation}
where the integration path is some vertical line. This two-stage Mellin-Cauchy  formula  is the beginning point of  the Depoissonization path. 

\medskip 
We now compare  the formula obtained by this two stage approach with the  Mellin-Newton-Rice formula. 
 First remark that, as the function $P(z) e^{z}$ is entire, we can replace the contour $\{|z| = r\}$ in \eqref{depoi-ii} by a Hankel contour  ${\cal H}$ starting  at $-\infty$ in the upper halfplane, winding clockwise around the origin and proceeding towards $- \infty$ in the lower halfplane.  Then \eqref{depoi-ii} becomes 
 \begin{equation} \label {Han}
  f(n) = \frac {n!}{2 i \pi} \int_{{\cal H}}  P(z)\,  e^z  \frac 1 {z^{n+1}}  dz
  \end{equation}
Now, if we {\em formally substitute }\eqref{Me} into  \eqref{Han},  {\em interchange} the order of integration and use the equality 
$$ \frac 1 { \Gamma (n+1-s)} = \frac 1 {2i\pi}   \int_{{\cal H}}  e^z   \frac {z^{s}} {z^{n+1}}  dz \, , $$
we obtain the representation
\begin{equation} \label {tot}
  f(n) =  \frac {n!}{2 i \pi}  \int_{\uparrow} P^\ast (-s) \frac 1 {\Gamma(n+1-s)} ds \, 
\end{equation}
and we  recognize in \eqref{tot} the Rice integral 
$$   f(n) = \frac {n!}{2 i \pi}  \int_{\uparrow} \frac { P^\ast (-s)}{ \Gamma(-s)} \frac {\Gamma(-s)} {\Gamma(n+1-s)} ds   =  \frac {1}{2 i \pi}  \int_{\uparrow}   \Pi[f](s)  \frac { (-1) ^{n+1}\,  n!} {s(s-1)     \ldots (s-n)} ds \, .$$

\medskip
 This exhibits   a formal comparison between the two paths.  However,  this comparison is {\em only formal}  because the previous  manipulations  may be meaningless due to the divergence of the
integrals.

\end{document}